\documentclass[11pt]{amsart}

\usepackage{amsmath,amssymb,amsthm,enumitem,hyperref}
\usepackage[margin=1.05in]{geometry}
\usepackage[T1]{fontenc}
\usepackage{lmodern}
\usepackage{microtype}

\hypersetup{colorlinks=true, linkcolor=blue, citecolor=blue, urlcolor=blue}

\newtheorem{theorem}{Theorem}[section]
\newtheorem{lemma}[theorem]{Lemma}
\newtheorem{proposition}[theorem]{Proposition}

\theoremstyle{definition}
\newtheorem*{definition}{Definition}

\newtheorem*{problem}{Problem}
\newtheorem*{example}{Example}
\theoremstyle{remark}
\newtheorem{remark}{Remark}[section]

\usepackage{soul}

\newcommand{\C}{\mathcal C}
\newcommand{\D}{\mathcal D}
\newcommand{\B}{\mathcal B}
\newcommand{\R}{\mathcal R}
\newcommand{\F}{\mathcal F}
\newcommand{\res}{\operatorname{res}}
\newcommand{\wt}{\operatorname{wt}}

\title[A problem of Andrews and Dhar on partitions]{A problem of Andrews and Dhar on partitions}
\author{Simon Mahns, Ken Ono and Jujian Zhang}

\address{Axiom Math, 124 University Avenue, Palo Alto, CA 94301}
\email{simon@axiommath.ai}
\email{ken@axiommath.ai}
\email{jujian@axiommath.ai}

\date{\today}

\begin{document}
\begin{abstract}
This paper is motivated by a broad question about AI-assisted mathematics: can an AI system help discover and certify an explicit bijection between two infinite sequences of complicated combinatorial sets already known to be equinumerous?  The challenge is to find a reversible structure explaining that equality uniformly across the sequence. We give an affirmative test case in the setting of a partition problem.
Andrews and Dhar introduced two partition families \(\C_3(n)\) and \(\D_3(n)\), and   for
``nonexceptional'' \(n\), they asked for a bijective proof of their equality
\[
        |\C_3(n)|=\frac{|\D_3(n)|}{3}.
\]
We prove a residue-class equidistribution
theorem for \(\D_3(n)\) that identifies a ``canonical third'' subset
\(\D_3^{(0)}(n)\subseteq \D_3(n)\). Answering their question, we construct a bijection
\[
        \iota_n: \C_3(n)\longrightarrow \D_3^{(0)}(n)
\]
as  a highly structured composition of four maps. AxiomProver autonomously produced and
Lean-verified the equidistribution theorem. The bijection was found
through human--AxiomProver collaboration, and the theorem was formalized and verified
through the collaboration.
 \end{abstract}



\maketitle

\section{Introduction}

Glaisher's partition theorem says the following.  Fix an integer $m\ge2$.  Then, for every $n\ge0$, the number of partitions of $n$ in which no part is repeated $m$ or more times is equal to the number of partitions of $n$ into parts not divisible by $m$.  When $m=2$, this is Euler's theorem equating partitions into distinct parts with partitions into odd parts.  The classical bijection is the base $m$ carrying map: repeated equal parts are bundled in groups of $m$, and those bundles are carried upward by multiplying the part by $m$.

Andrews and Dhar recently found a companion to Glaisher's theorem in which the largest part on one side is transformed into a smallest-part condition on the other side.  Their cubic case is especially striking.  They define a family $\C_3(n)$ of partitions into positive integers whose largest part is divisible by $3$, say $3J$, and whose parts $\leq J$ occur with multiplicity at most two, and a family $\D_3(n)$ of partitions into nonnegative parts whose smallest part occurs exactly three times and whose larger parts occur at most twice.  For nonexceptional positive integers $n$ (i.e. integers not of the form
$T_k+1$ with $T_k:=k(k+1)/2$), they proved that
\[
        |\C_3(n)|=|\D_3(n)|/3.
\]
They asked the following problem (see Section~4 of \cite{AndrewsDhar}).

\begin{problem} For $n$ not one more than a triangular number, provide a bijective proof of the equality
\[
|\C_3(n)|=|\D_3(n)|/3.
\]
\end{problem}

We address this problem by first explaining the factor $3$, which is accounted for by a residue statistic on the $\D_3$-objects.  Namely, if $\mu\in \D_3(n)$, let $\tau(\mu)$ denote the number of parts of $\mu$ strictly larger than its smallest part.  We construct an explicit bijection from the $\C_3$-objects to the residue-zero class of $\tau(\mu)$ modulo $3$.  The equality of the three residue classes is proved by a root-of-unity generating-function calculation, so the resulting proof is an explicit bijection to a canonical third of $\D_3(n)$.  

To make this precise, 
we write a partition as a weakly decreasing finite list of nonnegative integers. Parts of size $0$ are allowed for $\D_3(n)$, while $\C_3(n)$ only consists of positive integer parts.  The \emph{weight} of a partition is the sum of its parts. Given a partition
\[
\lambda=(\lambda_1,\lambda_2,\ldots,\lambda_r),
\qquad
\lambda_1\geq \lambda_2\geq \cdots \geq \lambda_r,
\]
its \emph{Ferrers diagram} is the left-justified array of boxes with \(\lambda_i\)
boxes in the \(i\)-th row.  For example, the partition \(\lambda=(5,3,3,1)\)
has Ferrers diagram
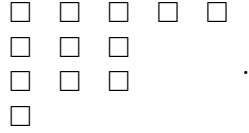
\[
\begin{array}{ccccc}
\Box&\Box&\Box&\Box&\Box\\
\Box&\Box&\Box\\
\Box&\Box&\Box\\
\Box
\end{array}.
\]
The \emph{conjugate partition} \(\lambda'\) is obtained by interchanging rows and
columns in this diagram, or equivalently by reflecting the diagram across its
main diagonal.  Therefore, the \(i\)-th part of \(\lambda'\) is the number of rows of
the Ferrers diagram of \(\lambda\) that contain at least \(i\) boxes:
\[
\lambda'_i=\#\{j:\lambda_j\ge i\}.
\]
For instance, for \(\lambda=(5,3,3,1)\), we have
\[
\lambda'=(4,3,3,1,1),
\]
because the original diagram has four rows of length at least \(1\), three rows
of length at least \(2\), three rows of length at least \(3\), one row of length
at least \(4\), and one row of length at least \(5\).  Parts of size \(0\) do
not contribute boxes to the Ferrers diagram, so they do not affect the conjugate
partition.

We now turn to the explicit partition functions considered by Andrews and Dhar \cite{AndrewsDhar}.

\begin{definition}[The Andrews--Dhar partition families]\label{def:families}
Suppose that $n$ is a positive integer.

\smallskip
\noindent
(a) Let $\C_3(n)$ be the set of partitions $\lambda$ of $n$ such that
\begin{enumerate}[label=(\roman*)]
\item the largest part of $\lambda$ is divisible by $3$, say it is $3J$;
\item every part of $\lambda$ which is at most $J$ occurs at most twice.
\end{enumerate}
Let $C_3(n):=|\C_3(n)|$.

\smallskip
\noindent
(b) Let $\D_3(n)$ be the set of partitions $\mu$ of $n$ into nonnegative parts such that
\begin{enumerate}[label=(\roman*)]
\item the smallest part of $\mu$ occurs exactly three times;
\item every part strictly larger than the smallest part occurs at most twice.
\end{enumerate}
Let $D_3(n):=|\D_3(n)|$.
For later use in generating functions, we also set
\[
        \D_3(0):=\{(0,0,0)\},\qquad D_3(0):=1.
\]
\end{definition}

Parts of size $0$ are allowed only in the $D_3$ case.  They are never allowed in $\C_3(n)$.  Except when $\D_3(0)$ is explicitly used, all partitions outside the $D_3$ case have positive parts.  In $\D_3(n)$ the condition ``the smallest part occurs exactly three times'' determines exactly how zeros may occur: a $D_3$-partition either has no zero parts at all, or has exactly three zero parts.  In the latter case those three zeros are the three smallest parts, and every positive part is strictly larger than the smallest part and therefore occurs at most twice.  No partition counted by $D_3(n)$ has one zero, two zeros, or more than three zeros.

\begin{example}
For example, $(9,4)\in\C_3(13)$: its largest part is $9=3\cdot3$, and the parts at most $3$ occur at most twice.  Also, we have
\[
       (4,3,3,1,1,1)\in\D_3(13),
\]
because the smallest part $1$ occurs exactly three times and the larger parts $4,3,3$ have multiplicity at most two.
\end{example}

To isolate the denominator $3$ in $|\D_3(n)|/3$, we use the following residue statistic.
For $\mu\in\D_3(n)$, we define
\[
   \tau(\mu):=\#\{\text{parts of }\mu\text{ strictly larger than its smallest part}\}.
\]
For $i\in\{0,1,2\}$, set
\[
   \D_3^{(i)}(n):=\{\mu\in\D_3(n):\tau(\mu)\equiv i\pmod3\},
   \qquad
   D_3^{(i)}(n):=|\D_3^{(i)}(n)|.
\]

Our first result shows that $\tau(\mu)\pmod 3$ divides the $\D_3$ partitions into three disjoint sets of equal size when $n$ is not one more than a triangular number.

\begin{theorem}[Residue-class equidistribution]\label{thm:thm1}
Let $T_r:=r(r+1)/2$ be the $r$th triangular number.  If $n\ge1$ and $n\ne T_r+1$ for every integer $r\ge0$, then
\[
        D_3^{(0)}(n)=D_3^{(1)}(n)=D_3^{(2)}(n)=\frac{D_3(n)}3.
\]
\end{theorem}

\subsection*{The partition map \texorpdfstring{$\iota_n$}{iota-n}}

To answer the Andrews-Dhar problem, thanks to Theorem~\ref{thm:thm1}, it suffices to construct  a bijection between $\C_3(n)$ and $\D_3^{(0)}(n)$.
Here we define the relevant map
\begin{equation}
\iota_n: \C_3(n) \longrightarrow \D_3^{(0)}(n)
\end{equation}
by composing some partition maps, which turns out to be the desired bijection.

First, let $\B_3^{(2)}(N)$ be the set of partitions of $N$ into positive parts not divisible by $3$ whose largest part is congruent to $2$ modulo $3$. In  particular, we have $\B_3^{(2)}(0)=\varnothing$.  The finite ``Glaisher map''
\begin{equation}
        \Gamma:\C_3(n)\longrightarrow \B_3^{(2)}(n-1)
\end{equation}
is defined as follows.  Given $\lambda\in\C_3(n)$, let its largest part be $3J$.  Replace one copy of this largest part by $3J-1$.  For every other original part $t$, write uniquely
\[
       t=3^a u,\qquad 3\nmid u,
\]
and replace that one part $t$ by $3^a$ copies of $u$.  After sorting, the result is $\Gamma(\lambda)$.  The special part \(3J-1\) is a largest part of the output and is congruent
to \(2\) modulo \(3\). Indeed, every other original part is at most \(3J\);
parts strictly smaller than \(3J\) become parts at most \(3J-1\), while any
remaining copies of \(3J\) are converted into parts at most \(J\).

Next we need another partition map.  It is the modulus $3$ specialization of Stockhofe's direct bijection between $m$-flat and $m$-regular partitions \cite{Stockhofe}.  The relevant facts about this map were previously derived by Stockhofe.  We give a self-contained proof of the special case needed here, so that the construction and its inverse can be checked directly without having to appeal to the original source.

A positive partition $\alpha=(\alpha_1,\ldots,\alpha_r)$ is called \emph{$3$-flat} if
\[
        \alpha_i-\alpha_{i+1}<3\quad(1\le i<r),
        \qquad\text{and}\qquad
        \alpha_r<3.
\]
Equivalently, after appending a final zero, all consecutive gaps are $0$, $1$, or $2$.  A partition is \emph{$3$-regular} if none of its parts is divisible by $3$.  The nonzero residue sequence $\res(\alpha)$ is obtained from $\alpha$ by deleting all parts divisible by $3$ and recording the residues modulo $3$ of the remaining parts, in order.  Let $\F^{(2)}(N)$ be the set of $3$-flat partitions of $N$ whose first nonzero residue is $2$; in particular, $\F^{(2)}(0)=\varnothing$.

For a sequence $v=(v_1,\ldots,v_k)$ with $v_i\in\{1,2\}$, its residue core $\Lambda(v)$ is the unique $3$-flat, $3$-regular partition whose $i$th part is congruent to $v_i$ modulo $3$.  It is constructed from right to left: set $c_k=v_k$, and, once $c_{i+1}$ is known, choose $c_i$ to be the unique number among $c_{i+1},c_{i+1}+1,c_{i+1}+2$ congruent to $v_i$ modulo $3$.  Then $\Lambda(v)=(c_1,\ldots,c_k)$.  For the empty sequence, $\Lambda(\varnothing)$ is the empty partition.

A part $A_i$ of a $3$-flat partition $A=(A_1,\ldots,A_r)$ is \emph{flat-removable} if $A_i$ is divisible by $3$ and deleting that one part leaves a $3$-flat partition.  The next map removes the multiples of $3$ from a $3$-flat partition and records the lost information in an ordinary partition.  Flat-removable multiples can simply be deleted.  The remaining multiples require a compensated deletion, in which the larger parts are lowered by $3$ so that $3$-flatness is preserved.

We define the map
\begin{equation}
       \Phi_3:\{3\text{-flat partitions of }N\}
       \longrightarrow
       \{3\text{-regular partitions of }N\}
\end{equation}
with the following  algorithm.
\begin{enumerate}[label=\textbf{S\arabic*.},leftmargin=3.1em]
\item Start with $A=\alpha$ and an empty record list $\nu_{\rm rec}$.
\item Scan $A$ from smallest part to largest part.  Whenever the current part is flat-removable, delete it and append one third of its value to $\nu_{\rm rec}$.  The scan is dynamic: after a deletion, continue with the next part above the deleted part.
\item Scan the remaining partition from largest part to smallest part.  If the current part is $A_i=3a$, delete it, subtract $3$ from each of the $i-1$ larger parts, and append $a+i-1$ to $\nu_{\rm rec}$.
\item When all multiples of $3$ have been removed, the remaining partition is the residue core $\Lambda(\res(\alpha))$.  Sort the record list $\nu_{\rm rec}$ into a partition $\nu$, and set
\[
        \Phi_3(\alpha):=\Lambda(\res(\alpha))+3\nu',
\]
where the sum is componentwise and missing parts are treated as zero.
\end{enumerate}

The inverse to $\Phi_3$ is equally explicit.  Given a $3$-regular partition $\beta$, compute its residue sequence $v=\res(\beta)$ and its residue core $A=\Lambda(v)$.  The partitions $\beta$ and $A$ have the same length and the same ordered nonzero residue sequence; in particular, $A_i\equiv\beta_i\pmod3$ for every $i$.  Put
\[
        q_i:=\frac{\beta_i-A_i}{3}
\]
for every index $i$.  Lemma~\ref{lem:extract} proves that the list $q=(q_1,q_2,\ldots)$ is a partition; set $\nu:=q'$.  Now read the parts of $\nu$ from largest to smallest.  For a current part $p$, perform one of the following insertions.
\begin{itemize}[leftmargin=2.4em]
\item A \emph{hard insertion of size $p$} chooses an integer $h$ with $0\le h\le p-1$, not exceeding the current length, adds $3$ to the first $h$ parts of the current partition, and inserts a new part $3(p-h)$ immediately after those $h$ larger parts.  It is admissible if the result is $3$-flat and the inserted multiple of $3$ is not flat-removable.
\item An \emph{easy insertion of size $p$} inserts a new part $3p$ in weakly decreasing order.  It is admissible if the result is $3$-flat and the inserted part is flat-removable.
\end{itemize}
Lemmas~\ref{lem:exist-insertion}--\ref{lem:well-defined-inverse} prove that this reverse step is always available in the present situation and is unique: if an admissible hard insertion exists, it is the required move; otherwise the unique admissible move is the easy insertion.  After all parts of $\nu$ have been processed, the resulting $3$-flat partition is $\Phi_3^{-1}(\beta)$.  We shall verify that these two algorithms are inverse to each other.

The last operation is a smallest-part raising map
\[
        \mathcal T:\R(N)\longrightarrow \D_3^{(0)}(N+1),
\]
where $\R(N)$ is the set of positive partitions $\sigma$ of $N$ satisfying the following two conditions:
\begin{enumerate}[label=(\roman*),leftmargin=2.6em]
\item no part of $\sigma$ occurs three or more times;
\item either $\ell(\sigma)\equiv2\pmod3$, or $\ell(\sigma)\equiv0\pmod3$ and the smallest part of $\sigma$ is unique.
\end{enumerate}
The second alternative presupposes that $\sigma$ is nonempty; in particular, this convention gives $\R(0)=\varnothing$.
For $\sigma\in\R(N)$, define $\mathcal T(\sigma)$ as follows.  If $\ell(\sigma)\equiv2\pmod3$, adjoin a new part $1$.  If $\ell(\sigma)\equiv0\pmod3$, raise the unique smallest part by $1$.  Call the resulting positive partition $\eta$.  If the smallest part of $\eta$ occurs exactly three times, set $\mathcal T(\sigma)=\eta$.  Otherwise, we set
\[
        \mathcal T(\sigma)=\eta\cup(0,0,0).
\]
Therefore, the only zeros ever created by the forward map are exactly three zeros, and they are created only in the final step.

We may now write the bijection in one line.  For $\lambda\in\C_3(n)$, define
\begin{equation}
        \boxed{\;
        \iota_n(\lambda)
        :=\mathcal T\left(\left(\Phi_3^{-1}(\Gamma(\lambda))\right)'\right)
        \;} .
\end{equation}
Equivalently: apply the finite Glaisher map, undo $\Phi_3$, conjugate, and then apply the smallest-part raising map.

Theorem~\ref{thm:thm1} explains why the three residue classes in $\D_3(n)$ have the same cardinality in the nonexceptional cases.  We now prove the main combinatorial refinement: $\C_3(n)$ is naturally bijective with the residue-zero class.

\begin{theorem}[The bijection]\label{thm:thm2}
For every $n\ge1$, the map just defined is a bijection
\[
        \iota_n:\C_3(n)\longrightarrow \D_3^{(0)}(n).
\]
In particular, if $n$ is not one more than a triangular number, then Theorem~\ref{thm:thm1} gives
\[
        |\C_3(n)|=|\D_3^{(0)}(n)|=|\D_3^{(1)}(n)|=|\D_3^{(2)}(n)|=\frac{D_3(n)}3.
\]
\end{theorem}

\begin{example}[The complete case $n=6$]\label{ex:n6}
Write $0^3$ for the three-part string $(0,0,0)$.  The $C_3$-objects of weight $6$ are
\[
        \C_3(6)=\{(6),\ (3,3),\ (3,2,1)\}.
\]
The $D_3$-objects of weight $6$ are
\[
\begin{aligned}
\D_3(6)=\{& (6,0^3),\ (5,1,0^3),\ (4,2,0^3),\ (4,1,1,0^3),\ (3,3,0^3), \\
          & (3,2,1,0^3),\ (2,2,1,1,0^3),\ (3,1,1,1),\ (2,2,2)\}.
\end{aligned}
\]
Among these, the residue-zero condition $\tau(\mu)\equiv0\pmod3$ selects
\[
        \D_3^{(0)}(6)
        =\{(4,1,1,0^3),\ (3,2,1,0^3),\ (2,2,2)\}.
\]
We now apply the map to each element of $\C_3(6)$.  In the table,
\[
        \rho=\Gamma(\lambda),\qquad
        \alpha=\Phi_3^{-1}(\rho),\qquad
        \sigma=\alpha'.
\]
Then we have $\iota_6(\lambda)=\mathcal T(\sigma)$.
\[
\begin{array}{|c|c|c|c|c|}\hline
\lambda & \rho & \alpha & \sigma & \iota_6(\lambda) \\
\hline
(6)     & (5)         & (3,2)       & (2,2,1) & (2,2,2) \\
(3,3)   & (2,1,1,1)   & (2,1,1,1)   & (4,1)   & (4,1,1,0^3) \\
(3,2,1) & (2,2,1)     & (2,2,1)     & (3,2)   & (3,2,1,0^3)\\ \hline
\end{array}
\]

For instance, the first row uses
\[
        (6)\xrightarrow{\Gamma}(5),
        \qquad
        \Phi_3^{-1}(5)=(3,2),
        \qquad
        (3,2)'=(2,2,1),
\]
and since $(2,2,1)$ has length $3\equiv0\pmod3$ and has a unique smallest part, the final step raises the $1$ to $2$, giving $(2,2,2)$.  The second and third rows have length $2\equiv2\pmod3$ after conjugation, so the final step first adjoins a $1$; because the resulting smallest part does not occur exactly three times, it then appends exactly three zeros.  Therefore, the image of $\C_3(6)$ is precisely $\D_3^{(0)}(6)$.

\end{example}

\begin{example}[The image of $(9,4)$]
We now spell out the example mentioned above.  Let
\(
        \lambda=(9,4)\in\C_3(13).
\)
The finite Glaisher step replaces the distinguished largest part by \(8\), and
since the remaining part \(4\) is already not divisible by \(3\), it gives
\[
        \rho=\Gamma(9,4)=(8,4).
\]
The inverse Stockhofe step gives
\[
        \alpha=\Phi_3^{-1}(8,4)=(5,3,3,1).
\]
Indeed, the residue sequence is \((2,1)\), its residue core is \((2,1)\), and
\((8,4)-(2,1)=3(2,1)\), so \(q=(2,1)\) and \(q'=(2,1)\).  The reverse
insertion of size \(2\) into \((2,1)\) is the hard insertion after one larger
part, producing \((5,3,1)\).  The subsequent reverse insertion of size \(1\) is
the easy insertion of \(3\), producing \((5,3,3,1)\).  Therefore
\[
        \sigma=(5,3,3,1)'=(4,3,3,1,1).
\]
Since \(\ell(\sigma)=5\equiv2\pmod3\), the raising map adjoins a new part
\(1\).  The resulting smallest part occurs exactly three times, so no zeros are
added.  Hence
\[
        \iota_{13}(9,4)=\mathcal T(\sigma)=(4,3,3,1,1,1)
        \in\D_3^{(0)}(13),
\]
because the smallest part \(1\) occurs exactly three times, the larger parts
\(4,3,3\) have multiplicity at most two, and \(\tau=3\equiv0\pmod3\).
\end{example}

\begin{remark}
After the completion of this work, we learned of a related manuscript of Liu \cite{LiuGlaisherCompanion}. Liu studies a related positive-part companion problem in general modulus \(m\). In his notation, the relevant partitions are those in which the smallest part occurs at most \(m\) times and all larger parts occur fewer than \(m\) times; the residue statistic is the total number of parts modulo \(m\). He proves, for \(m=3\), a root-of-unity equidistribution theorem away from the same shifted triangular exceptional set, and also gives a sign-reversing involution proof. He further constructs, for every \(m\geq 2\), a Stockhofe-type bijection between the residue-zero subfamily and \(m\)-regular partitions whose largest part is congruent to \(m-1 \pmod m\). Combined with a bijection of Lin and Zang, this gives a combinatorial proof of the Andrews--Dhar identity.

The present paper addresses the Andrews--Dhar formulation directly. Our bijection is a direct map
\[
\iota_n:C_3(n)\longrightarrow D_3^{(0)}(n),
\]
where
\[
D_3^{(0)}(n)=\{\mu\in D_3(n):\tau(\mu)\equiv 0 \pmod 3\}.
\]
The two papers share a Stockhofe flat-to-regular mechanism, but they apply it to different partition models and use different residue statistics. Liu's result gives a broad all-moduli companion framework, whereas the present work gives a direct bijective answer to the Andrews--Dhar cubic problem as stated.
\end{remark}

\section{Proof of Theorem~\ref{thm:thm1}}\label{sec:proof}
Here we prove Theorem~\ref{thm:thm1} using standard generating function calculations.
We require  the standard \(q\)-Pochhammer notation
\[
(a;q)_0:=1,\qquad
(a;q)_N:=\prod_{j=0}^{N-1}(1-aq^j)\quad (N\geq 1),
\]
and
\[
(a;q)_\infty:=\prod_{j=0}^{\infty}(1-aq^j).
\]
Thus, for example, we have
\[
(q;q)_s=(1-q)(1-q^2)\cdots(1-q^s).
\]
All generating-function identities below are interpreted as identities of
formal power series.

\begin{proof}[Proof of Theorem~\ref{thm:thm1}]

Let $\omega:=e^{2\pi i/3}$, and introduce the bivariate generating function
\[
        F(z;q):=\sum_{n\ge0}\sum_{\mu\in\D_3(n)}z^{\tau(\mu)}q^n.
\]
If the smallest part is $s\ge0$, then the three smallest parts contribute $q^{3s}$, and every larger part $r>s$ may occur zero, one, or two times, contributing $1+zq^r+z^2q^{2r}$.  Hence
\[
        F(z;q)=\sum_{s\ge0}q^{3s}\prod_{r\ge s+1}(1+zq^r+z^2q^{2r}).
\]
At $z=\omega$, the elementary factorization
\[
        1+\omega x+\omega^2x^2=(1-x)(1-\omega^2x)
\]
gives
\[
        F(\omega;q)=\sum_{s\ge0}q^{3s}(q^{s+1};q)_\infty(\omega^2q^{s+1};q)_\infty
        =(q;q)_\infty\sum_{s\ge0}\frac{q^{3s}}{(q;q)_s}(\omega^2q^{s+1};q)_\infty.
\]
Using Euler's expansion, which is a consequence of the $q$-binomial theorem, we have
\[
        (t;q)_\infty=\sum_{k\ge0}\frac{(-1)^kq^{k(k-1)/2}t^k}{(q;q)_k}
\]
and then summing the resulting geometric $q$-series by
\[
        \sum_{s\ge0}\frac{q^{(k+3)s}}{(q;q)_s}=\frac1{(q^{k+3};q)_\infty},
\]
we obtain
\[
\begin{aligned}
        F(\omega;q)
        = (q;q)_\infty\sum_{k\ge0}
        \frac{(-1)^k\omega^{2k}q^{k(k+1)/2}}{(q;q)_k(q^{k+3};q)_\infty} = \sum_{k\ge0}(-1)^k\omega^{2k}q^{k(k+1)/2}(1-q^{k+1})(1-q^{k+2}).
\end{aligned}
\]
Expanding the last two factors and shifting indices yields
\[
\begin{aligned}
F(\omega;q)
&=\sum_{k\ge0}(-1)^k\omega^{2k}q^{k(k+1)/2}
 +(1+q)\sum_{k\ge1}(-1)^k\omega^{2k-2}q^{k(k+1)/2}  \\
&\hspace{3.5em}+\sum_{k\ge2}(-1)^k\omega^{2k-4}q^{k(k+1)/2}.
\end{aligned}
\]
For every $k\ge2$, the coefficient of $q^{k(k+1)/2}$ in the three displayed sums is
\[
        (-1)^k\omega^{2k-4}(\omega^4+\omega^2+1)=0.
\]
The triangular exponents $T_k$ with $k\ge2$ have therefore cancelled.  The only terms that survive, besides the constant term, are the $q\cdot q^{T_r}$ terms coming from the middle shifted sum.  Thus the remaining contribution from the factor $(1+q)$ is supported precisely at exponents one more than triangular numbers.  More explicitly,
\[
        F(\omega;q)=1+\sum_{r\ge0}(-1)^r\omega^{2r-2}q^{T_r+1}.
\]
The same calculation with $\omega$ replaced by $\omega^2$ gives
\[
        F(\omega^2;q)=1+\sum_{r\ge0}(-1)^r\omega^{r-1}q^{T_r+1},
\]
so both nontrivial cubic-root filters have zero coefficient in every positive degree that is not of the form $T_r+1$.

Now write $a_i(n):=D_3^{(i)}(n)$.  The coefficient of $q^n$ in $F(\omega;q)$ is
\[
        a_0(n)+\omega a_1(n)+\omega^2a_2(n),
\]
and the coefficient of $q^n$ in $F(\omega^2;q)$ is
\[
        a_0(n)+\omega^2 a_1(n)+\omega a_2(n).
\]
For nonexceptional $n$, both quantities vanish.  Therefore, we have that
\[
        a_i(n)=\frac13\left(D_3(n)+\omega^{-i}\cdot0+\omega^{-2i}\cdot0\right)=\frac{D_3(n)}3
\]
for $i=0,1,2$.  This proves the theorem.  
\end{proof}

\section{The map $\iota_n$}\label{iota-m}
Thanks to Theorem~\ref{thm:thm1}, the Andrews-Dhar problem is solved by the bijection in Theorem~\ref{thm:thm2}.
This bijection is assembled from four elementary-looking transformations.  The rest of the proof checks that each transformation has the exact source and target claimed in the introduction.
The map in question is given by  the  composition
\[
        \iota_n=\mathcal T\circ(\cdot)'\circ\Phi_3^{-1}\circ\Gamma.
\]
We prove that each factor is a bijection between the stated intermediate sets.

\subsection{The finite Glaisher step}
The first step is a finite version of Glaisher carrying in which one distinguished largest part is lowered by one.  The truncation at the largest part is the only point requiring a check.

For a fixed $J\ge1$, let $\C_{3,J}(n)$ be the subset of $\C_3(n)$ consisting of partitions whose largest part is $3J$.  Let $\B_{3,J}^{(2)}(n-1)$ be the subset of $\B_3^{(2)}(n-1)$ consisting of partitions whose largest part is exactly $3J-1$.

The next proposition records that this finite carrying procedure loses no information.

\begin{proposition}\label{prop:gamma}
The finite Glaisher map restricts to a bijection
\[
        \Gamma_J:\C_{3,J}(n)\longrightarrow\B_{3,J}^{(2)}(n-1).
\]
Consequently, we have
\[
        \Gamma:\C_3(n)\longrightarrow\B_3^{(2)}(n-1)
\]
is a bijection.
\end{proposition}

\begin{proof}
The construction of $\Gamma_J$ lowers the weight by $1$, since the distinguished part $3J$ is replaced by $3J-1$, while every other replacement preserves weight:
\[
        3^a u=\underbrace{u+u+\cdots+u}_{3^a\text{ copies}}.
\]
All produced parts are not divisible by $3$, and the distinguished part $3J-1$ is the largest part and is congruent to $2$ modulo $3$.  Hence $\Gamma_J$ lands in $\B_{3,J}^{(2)}(n-1)$.

We now construct the inverse.  Start with $\rho\in\B_{3,J}^{(2)}(n-1)$.  Replace one copy of the largest part $3J-1$ by $3J$.  For every $u$ with $3\nmid u$, let $M_u$ be the number of remaining copies of $u$, and put
\[
        E_u:=\max\{e\ge0:u3^e\le3J\}.
\]
Write $M_u$ uniquely as
\[
        M_u=Q_u3^{E_u}+\sum_{e=0}^{E_u-1}d_{u,e}3^e,
        \qquad Q_u\ge0,
        \qquad d_{u,e}\in\{0,1,2\}.
\]
Insert $Q_u$ copies of $u3^{E_u}$ and, for $0\le e<E_u$, insert $d_{u,e}$ copies of $u3^e$.  Do this independently for all $u$ not divisible by $3$, and sort.

The resulting partition has largest part $3J$.  It remains to check the multiplicity condition below $J$.  By the definition of $E_u$,
\[
        u3^{E_u}\le3J<u3^{E_u+1}.
\]
If $u3^{E_u}\le J$, then multiplying by $3$ gives $u3^{E_u+1}\le3J$, a contradiction.  Hence $u3^{E_u}>J$.  This top bucket may occur with arbitrary multiplicity $Q_u$.  Every lower bucket $u3^e$ with $e<E_u$ is at most $J$, and it occurs with multiplicity $d_{u,e}\in\{0,1,2\}$.  Thus every part at most $J$ occurs at most twice.

For each fixed $u$, the forward map replaces one copy of $u3^e$ by $3^e$ copies of $u$, and the inverse recovers the multiplicities by the displayed base-$3$ expansion with the top bucket separated.  The distinguished part is reversed by $3J\leftrightarrow 3J-1$.  Therefore the two maps are inverse bijections.
\end{proof}

\subsection{The raising map and its inverse}
The final step of the bijection changes a length congruence into a smallest-part condition.  The inverse is simple, but the proof below records why zeros occur only as an exact triple.

The map $\mathcal T$ was defined in the introduction.  Its inverse lowers one smallest part.

For $n\ge1$ and $\mu\in\D_3^{(0)}(n)$, define $L(\mu)$ as follows.  If the smallest part of $\mu$ is $0$, delete the three zero parts; if the smallest part is positive, do nothing.  Call the resulting positive partition $\eta$.  Let $s$ be the smallest part of $\eta$.  Lower one copy of $s$ to $s-1$, delete that new part if it is $0$, and sort.  The result is $L(\mu)$.

The following lemma proves both well-definedness and invertibility of the raising step.

\begin{lemma}\label{lem:L}
For $n\ge1$, the maps
\[
        L:\D_3^{(0)}(n)\longrightarrow\R(n-1),
        \qquad
        \mathcal T:\R(n-1)\longrightarrow\D_3^{(0)}(n)
\]
are inverse bijections.
\end{lemma}

\begin{proof}
Let $\mu\in\D_3^{(0)}(n)$.  If the smallest part of $\mu$ is positive, then before lowering, the number of positive parts is $3+\tau(\mu)$.  If the smallest part is $0$, then after deleting the three zeros, the number of positive parts is $\tau(\mu)$.  In both cases this length is divisible by $3$, because $\tau(\mu)\equiv0\pmod3$.

Lowering one smallest positive part lowers the weight by $1$.  If the original smallest part was positive, then three copies of it existed before lowering; after lowering, only two copies remain at that old size, and the newly created smaller part is either deleted because it is zero or occurs once.  All larger parts already occurred at most twice.  If the original smallest part was zero, then, after the three zeros are deleted, all remaining positive parts already occur at most twice.  Hence no part of $L(\mu)$ occurs three or more times.

There are two length cases.  If the lowered part was $1$, then it became $0$ and was deleted; the length changes from $0\pmod3$ to $2\pmod3$.  If the lowered part was greater than $1$, the length remains $0\pmod3$, and the new smallest part is unique.  Therefore $L(\mu)\in\R(n-1)$.

Now start with $\sigma\in\R(n-1)$.  If $\ell(\sigma)\equiv2\pmod3$, the map $\mathcal T$ adjoins a $1$, reversing exactly the case in which a lowered $1$ had been deleted.  If $\ell(\sigma)\equiv0\pmod3$, then the smallest part is unique by definition of $\R$, and $\mathcal T$ raises that part by $1$, reversing exactly the case in which a positive smallest part had been lowered.  After this reversal, either the smallest positive part occurs exactly three times, or it does not.  Apart from the part just adjoined or raised, all positive multiplicities are unchanged from $\sigma$ and hence remain at most two; the only possible creation of a triple positive part is therefore at the new smallest positive part.  If that part occurs exactly three times, the output is already in $\D_3^{(0)}(n)$.  If it does not, the only possible original $D_3$-partition had smallest part $0$, so $\mathcal T$ appends exactly three zeros.  In both cases the number of parts above the smallest part is divisible by $3$, so the output lies in $\D_3^{(0)}(n)$.

The two descriptions are exact reversals, so $L$ and $\mathcal T$ are inverse bijections.
\end{proof}

\subsection{Conjugation}
Conjugation translates multiplicity bounds into flatness bounds.  It also converts the length congruence in $\R(N)$ into the first-residue condition defining $\F^{(2)}(N)$.

The next lemma is the standard conjugation dictionary specialized to the present residue condition.

\begin{lemma}\label{lem:conj}
Conjugation gives a bijection
\[
        \R(N)\longrightarrow\F^{(2)}(N),
        \qquad \sigma\longmapsto\sigma'.
\]
\end{lemma}

\begin{proof}
Let $\sigma\in\R(N)$ and set $\alpha=\sigma'$.  In conjugate partitions, multiplicities in $\sigma$ become consecutive differences in $\alpha$.  Since no part of $\sigma$ occurs three or more times, every gap in $\alpha$, including the final gap down to zero, is $0$, $1$, or $2$.  Thus $\alpha$ is $3$-flat.

If $\ell(\sigma)\equiv2\pmod3$, then $\alpha_1=\ell(\sigma)\equiv2\pmod3$, so the first nonzero residue of $\alpha$ is $2$.  If $\ell(\sigma)\equiv0\pmod3$ and the smallest part of $\sigma$ is $s$ and is unique, then
\[
        \alpha_1=\cdots=\alpha_s=\ell(\sigma)\equiv0\pmod3.
\]
Because the smallest part is unique, there is at least one part of $\sigma$ strictly larger than $s$; otherwise the length would be $1$, not a positive multiple of $3$.  Hence the next column exists and has height
\[
        \alpha_{s+1}=\ell(\sigma)-1\equiv2\pmod3.
\]
Thus again the first nonzero residue is $2$.

Conversely, let $\alpha\in\F^{(2)}(N)$ and set $\sigma=\alpha'$.  Since $\alpha$ is $3$-flat, all multiplicities in $\sigma$ are at most two.  If $\alpha_1\equiv2\pmod3$, then $\ell(\sigma)=\alpha_1\equiv2\pmod3$, so $\sigma\in\R(N)$.  Otherwise, the initial parts of $\alpha$ are divisible by $3$ until the first nonzero residue $2$ appears.  Suppose this occurs after $s$ initial parts.  Then
\[
        \alpha_s\equiv0\pmod3,
        \qquad
        \alpha_{s+1}\equiv2\pmod3.
\]
Because $\alpha$ is $3$-flat, the difference $\alpha_s-\alpha_{s+1}$ is $1$ or $2$; its residue modulo $3$ is $1$, so it is exactly $1$. 
For \(j<s\), the consecutive parts \(\alpha_j\) and \(\alpha_{j+1}\) are both
divisible by \(3\) and differ by less than \(3\), hence they are equal.
 This difference is the multiplicity of the smallest part $s$ of $\sigma$.  Hence the smallest part of $\sigma$ is unique and $\ell(\sigma)=\alpha_1\equiv0\pmod3$.  Therefore $\sigma\in\R(N)$.
\end{proof}

\subsection{The map \texorpdfstring{$\Phi_3$}{Phi-3}}
This is the technical heart of the paper.  We isolate the modulus-three Stockhofe map, prove that the deletion and insertion algorithms are well-defined, and then prove directly that they are inverse to one another.

This subsection verifies facts about $\Phi_3$ stated in the introduction.  As
noted above, these are the modulus $3$ case of Stockhofe's bijection.  The proofs
are included here to keep the argument self-contained.

The residue core is the minimal flat partition with a prescribed residue pattern.  We first prove that it is well-defined and unique.

\begin{lemma}\label{lem:core}
For every sequence $v=(v_1,\ldots,v_k)$ with $v_i\in\{1,2\}$, the residue core $\Lambda(v)$ exists and is unique.
\end{lemma}

\begin{proof}
The construction given in the introduction produces a weakly decreasing sequence because each $c_i$ is chosen from $c_{i+1},c_{i+1}+1,c_{i+1}+2$.  The differences are less than $3$, the last part is $1$ or $2$, and the residues are prescribed.  Hence it is a $3$-flat, $3$-regular partition with residue sequence $v$.

For uniqueness, the last part must be $v_k$, since it is positive, less than $3$, and congruent to $v_k$ modulo $3$.  Once $c_{i+1}$ is known, the previous part must lie among $c_{i+1},c_{i+1}+1,c_{i+1}+2$ and have residue $v_i$ modulo $3$.  There is exactly one such number.  Thus the partition is forced.
\end{proof}

The forward algorithm begins by deleting exactly those multiples of $3$ whose deletion preserves flatness.  The next lemma gives a local criterion for this condition.

\begin{lemma}\label{lem:removable}
Let $A=(A_1,\ldots,A_r)$ be $3$-flat and suppose $A_i=3a$.  Then this
indexed occurrence of $A_i$ is flat-removable exactly in the following cases:
\begin{enumerate}[label=(\roman*),leftmargin=2.6em]
\item $i=1$;
\item $i>1$ and $A_{i-1}=A_i$;
\item $i<r$ and $A_i=A_{i+1}$;
\item $1<i<r$ and the inequalities $A_{i-1}>A_i>A_{i+1}$ hold, with
\[
        A_{i-1}=3a+j_1,
        \qquad
        A_{i+1}=3(a-1)+j_2,
        \qquad
        0<j_1<j_2<3.
\]
\end{enumerate}
Consequently, after all flat-removable multiples have been deleted, every
remaining multiple $3a$ has both neighbors, neither neighbor is equal to it, and
those neighbors have the form
\[
        A_{i-1}=3a+j_1,
        \qquad
        A_{i+1}=3(a-1)+j_2,
        \qquad
        0<j_2\le j_1<3.
\]
\end{lemma}

\begin{proof}
Deleting $A_i$ can only create a new violation at the place where the part above
$A_i$ becomes adjacent to the part below $A_i$; all other consecutive gaps are
unchanged.  If $i=1$, there is no upper neighbor, so deletion cannot create a
new upper-lower gap.  If $i>1$ and $A_{i-1}=A_i$, then after deletion the upper
neighbor is equal to the deleted part, and the new gap is one of the old gaps.
Similarly, if $i<r$ and $A_i=A_{i+1}$, then the lower neighbor is equal to the
deleted part, and the new gap is again one of the old gaps.  These three cases
are therefore flat-removable.

It remains to consider the strict interior case
$A_{i-1}>A_i>A_{i+1}$.  Since $A$ is $3$-flat and $A_i=3a$, the neighboring
parts must be
\[
        A_{i-1}=3a+j_1,
        \qquad
        A_{i+1}=3(a-1)+j_2,
        \qquad j_1,j_2\in\{1,2\}.
\]
After deleting $3a$, the new gap is
\[
        (3a+j_1)-\bigl(3(a-1)+j_2\bigr)=3+j_1-j_2.
\]
This is less than $3$ if and only if $j_1<j_2$.  This proves the characterization.

For the final assertion, let $3a$ be a multiple which remains after Stage S2.  It
cannot be the first part, and it cannot be equal to either neighbor, since those
are flat-removable cases.  It also cannot be the last part, because a positive
multiple of $3$ cannot be the last part of a $3$-flat partition: the final gap to
zero would be at least $3$.  Hence it is a strict interior multiple.  The
strict-interior deletion condition just proved must fail, so its neighboring
residues satisfy $j_2\le j_1$.
\end{proof}

The record list must fit over the residue core when the final partition $\Lambda(\res(\alpha))+3\nu'$ is formed.  The next bound supplies precisely this length control.

\begin{lemma}\label{lem:record-bound}
Let $\alpha$ be a $3$-flat partition, and let $k$ be the number of parts of
$\alpha$ not divisible by $3$.  Every entry recorded by the forward
$\Phi_3$-algorithm is at most $k$.  Consequently, if $\nu$ is the sorted record
partition, then $\ell(\nu')\le k$.
\end{lemma}

\begin{proof}
We use the following elementary consequence of $3$-flatness.  If a $3$-flat
partition contains a part at least $3a$, then for each
$b=0,1,\ldots,a-1$ it contains at least one part in the interval
$\{3b+1,3b+2\}$.  Otherwise the partition would have to pass from a part at
least $3b+3$ to a part at most $3b$ in one consecutive step, or in the final
step to zero, producing a gap at least $3$.  Thus the tail from level $3a$ down
to zero contains at least $a$ nonmultiples of $3$.

A Stage S2 deletion of a flat-removable part $3a$ records $a$, so the preceding
observation gives $a\le k$.

Now consider a Stage S3 deletion.  At the moment it is made, suppose the current
part is $3a$ and there are $h$ larger parts.  The scan in Stage S3 is from
largest to smallest, so all larger multiples of $3$ have already been removed;
subtracting $3$ from a nonmultiple part never makes it divisible by $3$.  Hence
these $h$ larger parts are all nonmultiples of $3$.  The tail beginning with the
current part $3a$ contains at least $a$ further nonmultiples by the observation
above.  Therefore the total number $k$ of nonmultiples is at least $h+a$.  The
recorded value in Stage S3 is exactly $a+h$, so it is also at most $k$.

Since the largest part of $\nu$ is at most $k$, the conjugate partition $\nu'$ has
length at most $k$.
\end{proof}
We now check that the deletion algorithm really produces a $3$-regular partition and preserves the relevant data.

\begin{lemma}\label{lem:phi-forward-well}
The forward algorithm is well-defined.  It sends a $3$-flat partition $\alpha$ to a $3$-regular partition of the same weight and preserves the ordered nonzero residue sequence.
\end{lemma}

\begin{proof}
Stage S2 of the algorithm is well-defined by the definition of flat-removable: every deletion made in that stage leaves the partition $3$-flat.

For Stage S3, take a remaining multiple $A_i=3a$.  By Lemma~\ref{lem:removable}, it has both neighbors and is locally between
\[
        A_{i-1}=3a+j_1,
        \qquad
        A_{i+1}=3(a-1)+j_2,
        \qquad
        0<j_2\le j_1<3.
\]
When we delete $A_i$ and subtract $3$ from all larger parts, all gaps among the larger parts remain unchanged.  The only new gap to check is the gap between the lowered predecessor and the successor:
\[
        (3a+j_1-3)-\bigl(3(a-1)+j_2\bigr)=j_1-j_2,
\]
which is $0$ or $1$.  Therefore, Stage S3 also preserves $3$-flatness.  At the end of the deletion stages no multiples of $3$ remain; the temporary partition is $3$-flat, $3$-regular, and has the same nonzero residue sequence as $\alpha$.  By Lemma~\ref{lem:core}, it is $\Lambda(\res(\alpha))$.

Let $k$ be the length of this residue core.  Lemma~\ref{lem:record-bound} gives $\ell(\nu')\le k$.  Hence the final componentwise sum
\[
        \Lambda(\res(\alpha))+3\nu'
\]
has no extra parts appended beyond the residue core.  Each of its $k$ parts is congruent to the corresponding nonzero residue-core part modulo $3$, so every part remains nondivisible by $3$.  Thus the forward algorithm lands in the set of $3$-regular partitions.

Only multiples of $3$ are deleted, and subtracting or adding $3$ to existing nonmultiples does not change their residues.  Hence the ordered nonzero residue sequence is preserved.

Finally, we check the weight.  If Stage S2 deletes $3a$, it records $a$, and the final term $3\nu'$ restores weight $3a$.  If Stage S3 deletes $3a$ in position $i$, then it also subtracts $3$ from the $i-1$ larger parts, so the temporary partition loses
\[
        3a+3(i-1)=3(a+i-1).
\]
The recorded value is $a+i-1$, so the final term $3\nu'$ restores exactly that weight.  Therefore the output has the same weight as $\alpha$.
\end{proof}

For the inverse algorithm, the first concern is that the componentwise difference between a $3$-regular partition and its residue core is a genuine partition after division by $3$.

\begin{lemma}\label{lem:extract}
In the $\Phi_3^{-1}$ algorithm, the quotient partition $q$ extracted from a $3$-regular partition $\beta$ is a partition.
\end{lemma}

\begin{proof}
Let $A=\Lambda(\res(\beta))$.  Each $q_i=(\beta_i-A_i)/3$ is a nonnegative integer.  Indeed, $A$ is componentwise minimal among weakly decreasing positive sequences with the same ordered residue sequence: from the last part upward, each part is the least positive lift of the required residue which is at least the part below it.  Therefore $\beta_i\ge A_i$ and $\beta_i\equiv A_i\pmod3$ for every $i$.

To see that the quotients are weakly decreasing, set $\beta_{k+1}=A_{k+1}=0$, where $k$ is the length of $\beta$.  For $1\le i\le k$, both $\beta_i-\beta_{i+1}$ and $A_i-A_{i+1}$ have the same residue modulo $3$, and $A_i-A_{i+1}$ is the representative of that residue in $\{0,1,2\}$.  Since $\beta_i\ge\beta_{i+1}$, the difference
\[
        (\beta_i-\beta_{i+1})-(A_i-A_{i+1})
\]
is a nonnegative multiple of $3$.  Dividing by $3$ gives $q_i\ge q_{i+1}$.  Hence $q$ is weakly decreasing.
\end{proof}

For the next three lemmas, if \(A=(A_1,\ldots,A_r)\) is \(3\)-flat, put
\[
        h_i(A):=A_i+3i\qquad (1\le i\le r).
\]
Since \(A_i-A_{i+1}\in\{0,1,2\}\), the sequence
\[
        h_1(A),h_2(A),\ldots,h_r(A)
\]
is strictly increasing.  For a positive integer \(p\), call \(A\)
\emph{\(p\)-safe} if no \(h_i(A)\) belongs to
\[
        \{3,6,\ldots,3p\}.
\]

The next lemma says that a safe intermediate partition always admits a next reverse move of the required size.

\begin{lemma}[Existence of one reverse insertion]\label{lem:exist-insertion}
Let \(A=(A_1,\ldots,A_r)\) be a \(3\)-flat partition, and let \(k\) be the number
of parts of \(A\) which are not divisible by \(3\).  If \(1\le p\le k\) and
\(A\) is \(p\)-safe, then at least one admissible insertion of size \(p\) exists.
More precisely:
\begin{enumerate}[label=(\roman*),leftmargin=2.6em]
\item if \(h_i(A)<3p<h_{i+1}(A)\) for some \(i\), then the hard insertion
after the first \(i\) parts is admissible;
\item if \(3p<h_1(A)\), then the easy insertion of \(3p\) in weakly decreasing
order is admissible.
\end{enumerate}
\end{lemma}

\begin{proof}
Because \(p\le k\le r\) and \(A_r\ge1\), we have
\[
        h_r(A)=A_r+3r\ge 1+3r\ge 1+3p>3p.
\]
Since \(A\) is \(p\)-safe, no \(h_i(A)\) is equal to \(3p\).  Hence either
\(3p<h_1(A)\), or there is a unique index \(i\) with
\[
        h_i(A)<3p<h_{i+1}(A).
\]

Assume first that \(h_i(A)<3p<h_{i+1}(A)\).  Then \(i<r\), and
\(h_i(A)<3p\) implies \(i\le p-1\), so the hard insertion after the first
\(i\) parts is an allowed hard insertion.  Put
\[
        m:=3(p-i).
\]
The new partition is
\[
        B=(A_1+3,\ldots,A_i+3,\ m,\ A_{i+1},\ldots,A_r).
\]
The only gaps which are not inherited from \(A\) are the two gaps adjacent to
\(m\).  Since
\[
        h_{i+1}(A)-h_i(A)=3-(A_i-A_{i+1})\le3
\]
and \(h_i(A)<3p<h_{i+1}(A)\), we have
\[
        3p-2\le h_i(A)\le 3p-1,\qquad
        3p+1\le h_{i+1}(A)\le 3p+2.
\]
Therefore, we have
\[
        0< A_i+3-m=h_i(A)+3-3p<3
\]
and
\[
        0< m-A_{i+1}=3p+3-h_{i+1}(A)<3.
\]
Thus \(B\) is \(3\)-flat.  The inserted part \(m\) is not flat-removable, because
deleting it would put \(A_i+3\) next to \(A_{i+1}\), producing the gap
\[
        (A_i+3)-A_{i+1}=(A_i-A_{i+1})+3\ge3.
\]
So the hard insertion is admissible.

Now assume that \(3p<h_1(A)\).  Insert \(3p\) in weakly decreasing order, and
let \(t\) be the number of parts of \(A\) which are at least \(3p\).  Since
\(A_r<3\le3p\), we have \(t<r\).  If \(t>0\), then \(A_t\ge3p\) and
\(A_{t+1}\le3p\).  The \(3\)-flatness of \(A\) gives
\[
        A_t-3p<3,\qquad 3p-A_{t+1}<3,
\]
for otherwise the adjacent gap \(A_t-A_{t+1}\) would be at least \(3\).  If
\(t=0\), then all parts of \(A\) are below \(3p\), and
\(3p<h_1(A)=A_1+3\) gives \(3p-A_1<3\).  Hence the partition obtained by
inserting \(3p\) is \(3\)-flat.  Deleting the inserted part gives back the
original \(3\)-flat partition \(A\), so the inserted part is flat-removable.
Thus the easy insertion is admissible.
\end{proof}

Existence is not enough: the inverse algorithm must not require choices.  The following lemma proves the required uniqueness.

\begin{lemma}[Uniqueness of one reverse insertion]\label{lem:unique-insertion}
Let \(A=(A_1,\ldots,A_r)\) be a \(3\)-flat partition and let \(p\ge1\).
There is at most one admissible hard insertion of size \(p\).  If an admissible
hard insertion exists, then no easy insertion of size \(p\) is admissible.  The
easy insertion, when admissible, has a unique resulting partition.
\end{lemma}

\begin{proof}
We first prove uniqueness of a possible hard insertion.  Put
\[
        c_i:=A_i+3i \qquad (1\le i\le r).
\]
Since \(A\) is \(3\)-flat, the sequence \(c_1,c_2,\ldots,c_r\) is strictly
increasing:
\[
        c_{i+1}-c_i=3-(A_i-A_{i+1})\in\{1,2,3\}.
\]
Suppose that a hard insertion of size \(p\) after the first \(i\) parts is
admissible.  The case \(i=0\) cannot occur, because then the inserted multiple
of \(3\) is the largest part and is flat-removable.  The case \(i=r\) cannot
occur either, because the inserted part would be the smallest part and would be
a positive multiple of \(3\), contradicting the final \(3\)-flat condition
that the smallest part be less than \(3\).

Thus \(1\le i<r\).  Let the inserted part be \(m=3(p-i)\).  The two new
boundary gaps must both be \(0\), \(1\), or \(2\).  Since the inserted multiple
is required not to be flat-removable, neither boundary gap can be \(0\); hence
\[
        0< A_i+3-m<3,
        \qquad
        0< m-A_{i+1}<3.
\]
In particular, we have
\[
        c_i<3p<c_{i+1}.
\]
Because the sequence \(c_i\) is strictly increasing, there is at most one index
\(i\) with this property.  Hence there is at most one admissible hard insertion.

We next show that an admissible hard insertion excludes an easy insertion.  If
the hard insertion occurs after \(i\) parts, then \(c_i<3p<c_{i+1}\).  If an
easy insertion of \(3p\) were possible after \(h\) parts, then either \(h<i\)
or \(h\ge i\).  If \(h<i\) and \(h\ge1\), then \(c_h<c_i<3p\), so
\[
        A_h=c_h-3h<3p-3h<3p,
\]
contradicting the requirement that the first \(h\) parts lie weakly above the
inserted part \(3p\).  If \(h=0\), then \(c_1\le c_i<3p\), so
\(A_1<3p-3\), and the new top gap from \(3p\) to \(A_1\) is at least \(3\).

The case \(h=r\) is impossible for an easy insertion, because then the inserted
part \(3p\) would be the smallest part and would violate the final \(3\)-flat
condition.  Thus, if \(h\ge i\), we have \(h<r\), and
\[
        A_{h+1}\le A_{i+1}<3(p-i)\le 3p-3,
\]
so the lower gap from \(3p\) to \(A_{h+1}\) is at least \(3\).  In all cases
the easy insertion would fail to be \(3\)-flat.

Finally, the easy insertion, if admissible, has only one possible resulting
partition: insert \(3p\) into the weakly decreasing order of \(A\).  If there
are equal parts, choosing a different position among those equal parts gives
the same partition.  Thus there is at most one easy insertion.
\end{proof}

The safety condition is designed to be stable under the unique admissible insertion.  This stability is what lets the inverse algorithm process an entire record list.

\begin{lemma}[Safety is preserved]\label{lem:safety-preserved}
Let \(A\) be \(p\)-safe.  If an admissible insertion of size \(p\) is performed,
then the resulting \(3\)-flat partition is again \(p\)-safe.
\end{lemma}

\begin{proof}
Let \(B\) be the partition after the insertion.

First suppose the insertion is hard after the first \(i\) parts.  For every old
part of \(A\), the corresponding value \(h_j=A_j+3j\) is increased by exactly
\(3\) in \(B\): the first \(i\) parts are raised by \(3\), and the remaining
old parts have their indices increased by \(1\).  The newly inserted part has
\(h\)-value
\[
        3(p-i)+3(i+1)=3(p+1).
\]
Thus a non-divisible \(h\)-value remains non-divisible, and a divisible
\(h\)-value \(3s\) with \(s>p\) becomes \(3(s+1)\).  The new \(h\)-value
\(3(p+1)\) is also larger than \(3p\).  Hence no \(h\)-value of \(B\) lies in
\(\{3,6,\ldots,3p\}\).

Now suppose the insertion is easy after the first \(t\) parts.  The old
\(h\)-values above the inserted part are unchanged.  The inserted part has
\(h\)-value
\[
        3p+3(t+1)=3(p+t+1)>3p.
\]
Every old part below the insertion point has its index increased by \(1\), so
its \(h\)-value is increased by \(3\).  Again no new \(h\)-value can enter
\(\{3,6,\ldots,3p\}\).  Thus \(B\) is \(p\)-safe.
\end{proof}

We can now iterate the one-step existence and uniqueness statements.  This gives a genuine inverse algorithm for every $3$-regular input.

\begin{lemma}[The inverse algorithm is well-defined]\label{lem:well-defined-inverse}
For every \(3\)-regular partition \(\beta\), the inverse insertion algorithm
stated in the introduction is well-defined.  At each step exactly one
admissible insertion is available and it is the insertion selected by the rule:
use the admissible hard insertion if it exists, and otherwise use the easy
insertion.
\end{lemma}

\begin{proof}
Let \(v=\res(\beta)\), let \(A=\Lambda(v)\), and let \(k\) be the length of
\(v\), equivalently the number of parts of the \(3\)-regular partition
\(\beta\).  By Lemma~\ref{lem:extract}, the extracted quotient sequence
\[
        q_i=\frac{\beta_i-A_i}{3}
\]
is a partition, and \(\nu=q'\) is therefore a weakly decreasing list.  Every
part of \(\nu\) is at most \(k\).

The starting partition \(A=\Lambda(v)\) is \(p\)-safe for every \(p\le k\):
indeed
\[
        h_i(A)=A_i+3i\equiv A_i\pmod3,
\]
and \(A_i\) is never divisible by \(3\).

We now process the parts of \(\nu\) from largest to smallest.  Suppose the next
part is \(p\).  All previously processed parts are at least \(p\).  By induction
on the number of processed parts, the current partition is \(p\)-safe: this is
clear at the start, and after an insertion of size \(q\ge p\),
Lemma~\ref{lem:safety-preserved} gives \(q\)-safety, which immediately implies
\(p\)-safety.  Insertions add only parts divisible by \(3\), so the number of
non-divisible parts remains \(k\).  Since \(p\le k\), Lemma~\ref{lem:exist-insertion}
gives the existence of an admissible insertion of size \(p\).  Lemma~\ref{lem:unique-insertion}
gives uniqueness in the precise sense required by the algorithm: there is at
most one hard insertion, a hard insertion excludes an easy insertion, and the
easy insertion has a unique result when it is admissible.

After performing the insertion, Lemma~\ref{lem:safety-preserved} supplies the
safety condition needed for the next step.  Therefore the algorithm is
well-defined for the whole list \(\nu\).
\end{proof}

Before comparing the forward and inverse algorithms, we record one bookkeeping fact about labeled insertions.  It is the invariant that prevents later insertions from changing the value or the accounting position of an earlier inserted multiple.

\begin{lemma}[Later hard insertions do not raise earlier labels]\label{lem:no-raise-labels}
During the inverse algorithm, label every inserted multiple of $3$ by the size $p$ of the insertion which created it.  Suppose a labeled multiple of size $p$ currently occupies position $j$ in the current $3$-flat partition $A$.  Then
\[
        A_j+3j>3p.
\]
Consequently, if a later insertion has size $q\le p$, then a hard insertion of size $q$ cannot include this labeled part among the parts which are raised by $3$.
\end{lemma}

\begin{proof}
At the moment an easy insertion of size $p$ creates a part in position $j$, the new part is $3p$, so its value of $A_j+3j$ is $3p+3j>3p$.  At the moment a hard insertion of size $p$ after $h$ parts creates a part, the new part is $3(p-h)$ and its position is $h+1$, so its value of $A_j+3j$ is
\[
        3(p-h)+3(h+1)=3(p+1)>3p.
\]
Now consider a later step, and assume the displayed inequality holds before that step.  An easy insertion does not change the value of any old part; it either leaves the old index unchanged or increases it by one, so the displayed inequality remains true.  For a later hard insertion of size $q\le p$ after $h$ parts, admissibility forces
\[
        A_h+3h<3q
\]
when $h>0$.  If the labeled part were among the first $h$ parts, then the strict increase of the sequence $A_i+3i$ would give
\[
        A_h+3h\ge A_j+3j>3p\ge3q,
\]
a contradiction.  Thus the labeled part is not raised.  It may have its index increased by one if the new hard part is inserted above it, but its value is unchanged, so the inequality remains true.  This proves the invariant by induction over the later insertions.
\end{proof}

It remains to prove that the insertion algorithm and the deletion algorithm undo each other in the same accounting system.  The following compatibility lemma is the point at which the proof uses labels.

\begin{lemma}[Compatibility with the deletion algorithm]\label{lem:compatibility}
Let $\beta$ be a $3$-regular partition, and construct
$\alpha=\Phi_3^{-1}(\beta)$ by the inverse insertion algorithm.  When the
forward $\Phi_3$-algorithm is applied to $\alpha$, its recorded multiset is
exactly the multiset of parts of $\nu=q'$ extracted from $\beta$.  More
precisely, every easy insertion of size $p$ is reversed by a Stage S2 deletion
recording $p$, and every hard insertion of size $p$ is reversed by a Stage S3
deletion recording $p$.
\end{lemma}

\begin{proof}
We use labeled multiples of $3$.  When the inverse algorithm inserts a multiple,
label it by its size $p$ and by its type, easy or hard.  Nonmultiples are left
unlabeled.  
For this bookkeeping only, if a newly inserted easy part is equal to
pre-existing parts, we place its label on the lowest occurrence among the equal
parts.  This convention does not change the underlying partition, but it agrees
with the smallest-to-largest scan in Stage S2. Thus, within any block of equal multiples of $3$, Stage S2 encounters the most
recent easy labels first.
Since the parts of $\nu$ are processed from largest to smallest, every later insertion after a label of size $p$ has size at most $p$.

\smallskip
\noindent
We first show that Stage S2 deletes exactly the easy-labeled parts and records their labels.  An easy insertion of size $p$ creates a part equal to $3p$, and by definition deleting that newly inserted part restores the previous $3$-flat partition.  By Lemma~\ref{lem:no-raise-labels}, no later hard insertion raises this easy-labeled part; later easy insertions also do not change its value.

It remains only to check that later insertions do not destroy the
flat-removability of an easy-labeled part.  Let the easy-labeled part have
label $p$, so that its value is $3p$.  We claim the following invariant.
After any number of later insertions, this part still has value $3p$, and
if it is deleted then the resulting partition is again $3$-flat.  Indeed,
by Lemma~\ref{lem:no-raise-labels}, no later hard insertion raises this
labeled part, and easy insertions never change the values of old parts.  A
later easy insertion has size $q\le p$, and hence inserts a part $3q\le 3p$.
If this inserted part becomes adjacent to the easy-labeled part, then after
removing the easy-labeled part the new adjacent gap is no larger than a gap
which was already present at the moment the later easy insertion was made.
A later hard insertion also cannot raise the easy-labeled part.  Its newly
inserted multiple is at most $3q\le 3p$; if it were inserted immediately
above the easy-labeled part, weak decrease would force equality, which would
amount to a hard insertion of size $p$ after no larger parts, impossible by
admissibility.  Thus a later hard insertion does not change the two neighbors
which are relevant for deleting the easy-labeled part.  Therefore the
easy-labeled part remains flat-removable, remains equal to $3p$, and is
deleted in Stage S2 with record $p$.

Conversely, a hard-labeled part is not deleted in Stage S2.  Suppose that a
hard-labeled part \(x\) was created by a hard insertion of size \(p\) after
\(h\) larger parts.  Thus \(x=3(p-h)\) at the moment of creation.  Let \(u\)
be the last of the \(h\) old parts which were raised by \(3\), and let \(v\)
be the first old part below the inserted part.  These are the witness parts
for the non-removability of \(x\): immediately after the insertion, deleting
\(x\) would place \(u\) next to \(v\), and the resulting gap is at least \(3\).
Both witness parts are nonmultiples of \(3\), and hence neither is ever deleted
by Stage S2 or Stage S3.

We shall use the following simple barrier observation.  Once a hard insertion
has created a witness pair \(u,v\), later insertions cannot decrease the gap
between these two witness parts.  Indeed, a later easy insertion changes no old
part.  A later hard insertion either raises both witnesses, raises only the
upper witness, or raises neither; it can never raise the lower witness without
also raising the upper witness, because the parts raised by a hard insertion
are an initial segment of the partition.  Hence the witness gap is preserved or
increased.  Later insertions may place additional labeled multiples between
\(u\) and \(v\), but they do not remove the original obstruction created by the
hard insertion.

Now run Stage S2 from smaller parts to larger parts.  Suppose, for
contradiction, that some hard-labeled part is deleted in Stage S2, and choose
the first such part \(x\) encountered by the scan.  At that moment no
hard-labeled part below \(x\) has been deleted.  Consider the witness pair
\(u,v\) attached to \(x\).  All easy-labeled multiples between \(u\) and \(v\)
which the Stage S2 scan has reached have already been removed, and removing
such parts cannot decrease the witness gap.  If no hard-labeled multiple
remains between \(u\) and \(v\), then deleting \(x\) exposes the witness gap,
which is still at least \(3\), so \(x\) is not flat-removable.

It remains to rule out the case in which hard-labeled multiples still lie
between \(u\) and \(v\).  Among all hard-labeled multiples in this interval,
choose one whose witness interval is minimal under inclusion.  The witness
interval of any hard-labeled multiple lying strictly between the two witnesses
of this chosen part is contained in the chosen witness interval: witness parts
are nonmultiples, and later insertions never move one nonmultiple past another.
Thus, by minimality, no hard-labeled multiple lies strictly between the two
witnesses of the chosen part.  All easy-labeled multiples inside this minimal
interval have already been removed before the chosen hard-labeled part could
become flat-removable, because Stage S2 scans from smaller parts to larger
parts and removes precisely such easy multiples when they are encountered.
Therefore deleting the chosen hard-labeled part would expose its own witness
gap, still of size at least \(3\).  Hence it is not flat-removable.  This
contradiction shows that no hard-labeled part is deleted during Stage S2.
Consequently, Stage S2 deletes exactly the easy-labeled parts, and after Stage
S2 the remaining multiples of \(3\) are exactly the hard-labeled parts.

\smallskip
We now consider Stage S3.  Let $x$ be a hard-labeled part created by an insertion of size $p$ after $h$ larger parts.  By Lemma~\ref{lem:no-raise-labels}, no later hard insertion raises $x$, so throughout the inverse construction its value remains
\[
        x=3(p-h).
\]
The $h$ parts raised at the creation of $x$ are unlabeled nonmultiples, and they remain above $x$ until $x$ is deleted.  Any later inserted multiple above $x$ is deleted earlier in the largest-to-smallest Stage S3 scan; any later inserted easy multiple has already been deleted in Stage S2.  Thus, when Stage S3 reaches $x$, the number of parts strictly larger than $x$ is exactly $h$.  The Stage S3 rule therefore records
\[
        \frac{x}{3}+h=(p-h)+h=p,
\]
deletes $x$, and subtracts $3$ from precisely the same $h$ nonmultiple parts that were raised when $x$ was inserted.  This is exactly the inverse of the hard insertion which created $x$.

The unlabeled parts are the nonmultiples in the residue core, and neither Stage S2 nor Stage S3 deletes nonmultiples.  Therefore every inserted label is recorded exactly once, with the same value, and no extra values are recorded.  The forward record multiset is exactly the multiset of parts of $\nu$.
\end{proof}
The preceding lemmas now assemble into the promised self-contained modulus-three Stockhofe bijection.

\begin{proposition}\label{prop:phi}
The map $\Phi_3$ is a weight-preserving bijection from $3$-flat partitions of $N$ to $3$-regular partitions of $N$.  Its inverse is the insertion algorithm stated in the introduction.  Moreover, we have
\[
        \res(\Phi_3(\alpha))=\res(\alpha).
\]
\end{proposition}

\begin{proof}
Lemma~\ref{lem:phi-forward-well} proves that the forward algorithm lands in the correct
set, preserves weight, and preserves the ordered nonzero residue sequence.  It
remains to justify the inverse.

Run the forward algorithm on a \(3\)-flat partition \(\alpha\).  At the end of
the deletion stages, all multiples of \(3\) have been removed.  The remaining
partition is \(3\)-flat, \(3\)-regular, and has residue sequence
\(\res(\alpha)\); by Lemma~\ref{lem:core}, it is \(\Lambda(\res(\alpha))\).
The deleted weight has been recorded in a list \(\nu_{\rm rec}\).  Let \(\nu\) be the partition obtained by sorting \(\nu_{\rm rec}\).  The final output is
\[
        \beta=\Lambda(\res(\alpha))+3\nu'.
\]
Lemma~\ref{lem:extract} shows that the inverse extraction recovers this same
core and this same sorted record list \(\nu\).

We now reverse the deletions, reading the sorted record list from largest to
smallest.  A Stage S3 deletion of a part \(3a\) with exactly \(h\) larger parts
records
\[
        p=a+h.
\]
It also subtracts \(3\) from exactly those \(h\) larger parts.  Its inverse is
therefore forced: add \(3\) back to those \(h\) larger parts and insert
\[
        3a=3(p-h)
\]
immediately after them.  This is precisely a hard insertion of size \(p\).  The
calculation in the proof of Lemma~\ref{lem:phi-forward-well} shows that this insertion
is \(3\)-flat, and deleting the inserted part would recreate the gap enlarged
by \(3\); hence the inserted multiple is not flat-removable.

A Stage S2 deletion of a part \(3p\) records \(p\) and changes no other part.
Its inverse is forced as well: reinsert \(3p\) in weakly decreasing order.  Since
the original deletion was flat-removable, this is an admissible easy insertion.

Lemma~\ref{lem:unique-insertion} shows that each such reverse step is unique:
the hard position, if it exists, is forced; and an admissible hard insertion
excludes an easy insertion.  If no hard insertion exists, the easy insertion is
the only possible admissible move.  Thus reversing the recorded deletions
reconstructs the original \(3\)-flat partition \(\alpha\).

Conversely, begin with an arbitrary \(3\)-regular partition \(\beta\).  Extract
\(A=\Lambda(\res(\beta))\), the quotient partition \(q=(\beta-A)/3\), and the record
list \(\nu=q'\).  Lemma~\ref{lem:well-defined-inverse} proves, without any
appeal to an external bijection, that the insertion algorithm is well-defined
for this arbitrary input and that exactly one admissible move is available at
each step.  Let \(\alpha\) be the resulting \(3\)-flat partition.

Each insertion adds one recorded part \(p\) and changes the weight by exactly
\(3p\): this is immediate for an easy insertion, and for a hard insertion after
\(h\) parts the weight gain is
\[
        3h+3(p-h)=3p.
\]
The ordered nonzero residue sequence is unchanged throughout the insertion
process, because only multiples of \(3\) are inserted and all other changes are
by multiples of \(3\).  Therefore, when the forward \(\Phi_3\)-algorithm is
run on \(\alpha\), Lemma~\ref{lem:phi-forward-well} gives a \(3\)-regular partition
with residue sequence \(\res(\beta)\) and total weight
\[
        \wt(A)+3|\nu|
        =\wt(A)+3|q|
        =\wt(\beta).
\]
By Lemma~\ref{lem:compatibility}, the forward records are exactly the inserted
record multiset \(\nu\).  Hence the forward output is
\[
        A+3\nu'=A+3q=\beta.
\]
Therefore, the forward and reverse algorithms are inverse bijections, and the
residue-preservation statement follows from Lemma~\ref{lem:phi-forward-well}.
\end{proof}

\subsubsection{The bridge from \texorpdfstring{$\D_3^{(0)}$}{D3-zero} to \texorpdfstring{$\B_3^{(2)}$}{B3-two}}
The preceding pieces now combine into a direct bridge from the residue-zero $D_3$-partitions to the $3$-regular partitions whose largest part has residue $2$.  This bridge is the inverse of the map used in the definition of $\iota_n$.

The bridge proposition packages the raising map, conjugation, and $\Phi_3$ into a single bijection.

\begin{proposition}\label{prop:theta}
The map
\[
        \Theta_n:\D_3^{(0)}(n)\longrightarrow\B_3^{(2)}(n-1),
        \qquad
        \Theta_n(\mu):=\Phi_3\bigl(L(\mu)'\bigr),
\]
is a bijection.
\end{proposition}

\begin{proof}
Let $\mu\in\D_3^{(0)}(n)$.  By Lemma~\ref{lem:L}, $L(\mu)\in\R(n-1)$.  By Lemma~\ref{lem:conj}, the conjugate $L(\mu)'$ lies in $\F^{(2)}(n-1)$.  Proposition~\ref{prop:phi} then gives a $3$-regular partition of $n-1$ with the same ordered nonzero residue sequence.  In particular, its largest part has residue $2$ modulo $3$, so $\Theta_n(\mu)\in\B_3^{(2)}(n-1)$.

The inverse is the map
\[
        \rho\longmapsto \mathcal T\left((\Phi_3^{-1}(\rho))'\right).
\]
Indeed, if $\rho\in\B_3^{(2)}(n-1)$, then its first nonzero residue is $2$.  Since $\Phi_3^{-1}$ preserves the ordered nonzero residue sequence, $\alpha:=\Phi_3^{-1}(\rho)$ lies in $\F^{(2)}(n-1)$.  Lemma~\ref{lem:conj} gives $\alpha'\in\R(n-1)$, and Lemma~\ref{lem:L} says that $\mathcal T$ reverses $L$.  Hence the displayed map is the inverse of $\Theta_n$.
\end{proof}

\section{The proof of Theorem~\ref{thm:thm2}}
At this point all intermediate maps have been checked.  The main bijection is therefore the composition of the finite Glaisher bijection with the inverse of the bridge just constructed. We now prove Theorem~\ref{thm:thm2}.

\begin{proof}[Proof of Theorem~\ref{thm:thm2}]
Proposition~\ref{prop:gamma} gives a bijection
\[
        \Gamma:\C_3(n)\longrightarrow\B_3^{(2)}(n-1).
\]
Proposition~\ref{prop:theta} gives a bijection
\[
        \Theta_n:\D_3^{(0)}(n)\longrightarrow\B_3^{(2)}(n-1).
\]
Therefore, we have
\[
        \iota_n:=\Theta_n^{-1}\circ\Gamma
        =\mathcal T\circ(\cdot)'\circ\Phi_3^{-1}\circ\Gamma
\]
is a bijection from $\C_3(n)$ to $\D_3^{(0)}(n)$.  Since $\D_3^{(0)}(n)\subseteq\D_3(n)$, it is also an injection from $\C_3(n)$ into $\D_3(n)$.
\end{proof}

\section{AxiomProver's autonomous Lean verification}\label{sec:AxiomProver}

Here we provide the context for this project as well as the protocol used for Lean
formalization and verification (see \cite{Mathlib2020, Lean}).
This paper is motivated by the broad question: can an AI system help discover and certify an explicit bijection between two infinite sequences of complicated combinatorial sets already known to be equinumerous?   We give an affirmative test case in the setting of a partition problem demonstrated by
Theorem~\ref{thm:thm1} and Theorem~\ref{thm:thm2}.  AxiomProver is an AI system currently under development at Axiom Math.

Theorem~\ref{thm:thm1} demonstrates that 
$$
D_3^{(0)}(n)=D_3^{(1)}(n)=D_3^{(2)}(n)
$$ 
for nonexceptional $n$. 
This theorem is proved through a standard manipulation of generating functions. It is well within reach for all experts in the field of $q$-series. AxiomProver independently generated and verified this theorem in Lean \cite{Lean}.

On the other hand, AxiomProver's proof of Theorem~\ref{thm:thm2} presents a more compelling case. Specifically, the Andrews-Dhar problem sought a bijective demonstration of the equality
$$
|\C_3(n)|=|\D_3(n)|/3
$$
for nonexceptional values of $n$.
As stated, this problem is ambiguous.  How does one establish a bijection with ``one third'' of a set without defining what that third actually is? Thankfully, Theorem~\ref{thm:thm1} provides three candidates. We set out  to establish a bijection
$$
\iota_n: \C_3(n) \longrightarrow \D_3^{(0)}(n).
$$
Theorem~\ref{thm:thm2} offers such an $\iota_n$ as a complicated composition of four partition maps, each of which is a bijection. 
The $\iota_n$ were discovered and verified by
AI-human collaboration, by means of guided search combined with smart prompting.
Indeed, as described below, AxiomProver's Lean verification of Theorem~\ref{thm:thm2} was split among four separate runs.

\subsection*{Process}
The formal proofs provided in this work were developed and verified using Lean~4.28.0.
Compatibility with earlier or later versions is not guaranteed due to the
evolving nature of the Lean 4 compiler and its core libraries.
The relevant files are all posted in the following repository:
\begin{center}
  \url{https://github.com/AxiomMath/andrews_dhar_problem}
\end{center}
The formalization consisted of five separate runs in parallel: one for
Theorem~\ref{thm:thm1}, and four for Theorem~\ref{thm:thm2}, the latter split
according to the four component bijections of $\iota_n$.  The runs are named
\texttt{thm1} and \texttt{thm2\_split1} through \texttt{thm2\_split4}, and they
formalize, respectively, Theorem~\ref{thm:thm1}, Proposition~\ref{prop:gamma}
(the finite Glaisher step), Lemma~\ref{lem:L} (the raising map),
Lemma~\ref{lem:conj} (conjugation), and Proposition~\ref{prop:phi} (the modulus
$3$ Stockhofe map $\Phi_3$).
For each run, the input files, collected under \texttt{input/<run>/}, were the following.
\begin{itemize}
\item \texttt{task.md}: a natural-language description of the problem to be formalized.
\item \texttt{Thm1WithProof.tex} or \texttt{Thm2WithProof.tex}: the \LaTeX{} statement and proof that the formalization was asked to follow.
\item \texttt{requirement.md}: additional constraints on the formalization, present for the \texttt{thm2\_split*} runs only, requiring each statement to be proved conditionally on Theorem~\ref{thm:thm1}; and, in the case of {\texttt{thm2\_split4}}, the main result in~\cite{Glaisher}.
\end{itemize}
Two further inputs were shared across all runs: a \texttt{.environment} file fixing the Lean version (\texttt{lean-4.28.0}), and \texttt{AndrewsDhar\_arXivPaper.tex}, the source paper of Andrews and Dhar \cite{AndrewsDhar}.
Given these files,
AxiomProver produced, for each run, the following output files, collected under \texttt{Bijection/<run>/}.
\begin{itemize}
 \item \texttt{problem.lean}: a translation of the problem statement into Lean.
 \item \texttt{solution.lean}: the formal, machine-checked solution in Lean.
\end{itemize}
For the runs \texttt{thm1} and \texttt{thm2\_split1} through \texttt{thm2\_split3},
both files were generated autonomously by AxiomProver. For the run
\texttt{thm2\_split4}, AxiomProver autonomously generated the formal problem
statement \texttt{problem.lean} together with a logically and structurally complete solution assuming the main result in \cite{Glaisher}.
The file \texttt{solution.lean} presented in the repository was obtained from that partial solution,
which an author then completed.

After the formal solutions were generated, the human authors wrote this paper
(without the use of AI) for human readers.
At first glance, the proofs found by AxiomProver may not resemble the narrative presented in this paper.
Turning a Lean file into a human-readable proof is difficult
because Lean is written as code for a type-checker.

\section*{Declaration of generative AI and AI-assisted technologies in the manuscript preparation process}
As described in the preceding Appendix,
AxiomProver (an AI tool under development)
was used to produce a formal proof of Theorem~\ref{thm:thm1} and Theorem~\ref{thm:thm2}.
The paper was written without AI.


\begin{thebibliography}{99}

\bibitem{AndrewsBook}
G. E. Andrews,
\emph{The Theory of Partitions},
Cambridge University Press, 1998.

\bibitem{AndrewsDhar}
G. E. Andrews and A. Dhar,
\emph{On Glaisher's partition theorem},
arXiv:2512.12346v3, 2026.

\bibitem{Glaisher}
J. W. L. Glaisher,
\emph{A theorem in partitions},
Messenger of Math.\ \textbf{12} (1883), 158--170.

\bibitem{Mathlib2020}
The mathlib Community,
The {L}ean mathematical library,
in \emph{Proceedings of the 9th ACM SIGPLAN International Conference on
Certified Programs and Proofs (CPP 2020)},
ACM, 2020.

\bibitem{Lean}
L.~de~Moura, S.~Kong, J.~Avigad, F.~van~Doorn, and J.~von~Raumer,
The {L}ean theorem prover (system description),
in \emph{Automated Deduction -- CADE-25},
Lecture Notes in Computer Science~9195, Springer, 2015, 378--388.

\bibitem{LiuGlaisherCompanion}
J.-C. Liu, \emph{A bijective approach to a companion of Glaisher's partition theorem}, preprint, 2026.


\bibitem{Stockhofe}
D. Stockhofe,
\emph{Bijektive Abbildungen auf der Menge der Partitionen einer nat\"urlichen Zahl},
Ph.D. thesis, Bayreuth. Math. Schr. \textbf{10} (1982), 1--59.



\end{thebibliography}
\end{document}